\documentclass[]{theclass}
\usepackage{mathtools}
\definecolor{orchid}{rgb}{0.85, 0.44, 0.84}

\begin{document}

\begin{frontmatter}

\titledata{The Pairing-Hamiltonian property\\ in graph prisms}{}

\authordata{Marién Abreu}
{Dipartimento di Matematica, Informatica ed Economia \\ Universit\`{a} degli Studi della Basilicata, Italy}
{marien.abreu@unibas.it}{The research that led to the present paper was partially supported by a grant of the group GNSAGA
of INdAM.}

\authordata{Giuseppe Mazzuoccolo}
{Dipartimento di Scienze Fisiche, Informatiche e Matematiche\\ Universit\`{a} degli Studi di Modena e Reggio Emilia, Italy}
{}{Corresponding author: giuseppe.mazzuoccolo@unimore.it}

\authordata{Federico Romaniello}
{Dipartimento di Matematica ``Giuseppe Peano"\\ Universit\`{a} di Torino, Italy}
{federico.romaniello@unito.it}{}

\authordata{Jean Paul Zerafa}
{St. Edward's College, Triq San Dwardu\\ Birgu (Citt\`{a} Vittoriosa), BRG 9039, Cottonera, Malta; \\
Faculty of Economics, Management \& Accountancy, and Faculty of Education\\ L-Universit\`{a} ta' Malta, Msida, MSD 2080, Malta}
{zerafa.jp@gmail.com}{}


\begin{abstract}
Let $G$ be a graph of even order, and consider $K_G$ as the complete graph on the same vertex set as $G$. A perfect matching of $K_G$ is called a pairing of $G$. If for every pairing $M$ of $G$ it is possible to find a perfect matching $N$ of $G$ such that $M \cup N$ is a Hamiltonian cycle of $K_G$, then $G$ is said to have the Pairing-Hamiltonian property, or PH-property, for short. In 2007, Fink [\emph{J. Combin. Theory Ser. B}, \textbf{97}] proved that for every $d\geq 2$, the $d$-dimensional hypercube $\mathcal{Q}_d$ has the PH-property, thus proving a
conjecture posed by Kreweras in 1996. 
In this paper we extend Fink's result by proving that given a graph $G$ having the PH-property, the prism graph $\mathcal{P}(G)=G \square K_2$ of $G$ has the PH-property as well. Moreover, if $G$ is a connected graph, we show that there exists a positive integer $k_0$
such that the $k^{\textrm{th}}$-prism of a graph $\mathcal{P}^k(G)$ has the PH-property for all $k \ge k_0$. 
\end{abstract}
\end{frontmatter}

\section{Introduction}
The problem of extending perfect matchings of a graph to a Hamiltonian cycle has been first considered by Las Vergnas \cite{LasVergnas} and H\"{a}ggkvist \cite{Hag} in the 1970s. They both proved Ore-type conditions which ensure that every perfect matching of a graph having some initial conditions can be extended to a Hamiltonian cycle. 

Some years later, Kreweras \cite{Kre} conjectured that any perfect matching of the hypercube $\mathcal{Q}_d$, $d\geq 2$, can be extended to a Hamiltonian cycle. This conjecture was proved in 2007 by Fink \cite{Fink}. Actually, he proved a stronger version of the problem. Given a graph $G$, let $K_G$ denote the complete graph on the same vertex set $V(G)$ of $G$. Fink shows that every perfect matching of $K_{\mathcal{Q}_d}$, and not only the perfect matchings of $\mathcal{Q}_d$, can be extended to a Hamiltonian cycle of $K_{\mathcal{Q}_d}$, by using only edges of $\mathcal{Q}_d$. More generally, for a graph $G$ of even order, a perfect matching of $K_G$ is said to be a \emph{pairing} of $G$. Given a pairing $M$ of $G$, we say that $M$ can be \emph{ extended } to a Hamiltonian cycle $H$ of $K_G$ if we can find a perfect matching $N$ of $G$ such that $M \cup N = E(H)$, where $E(H)$ is the set of edges of $H$.

A graph $G$ is said to have the \emph{Pairing-Hamiltonian property} (or, the PH-property for short), if every pairing $M$ of $G$ can be extended to a Hamiltonian cycle as described above. For simplicity, we shall also say that a graph $G$ is PH if it has the PH-property. This notation was introduced in \cite{AAAHST}, where amongst other results, a classification of cubic graphs that admit the PH-property was given: these are the complete graph $K_4$, the complete bipartite graph $K_{3,3}$, and the cube $\mathcal{Q}_3$. We remark that this was the first non-trivial classification of graphs (having regular degree) admitting the PH-property, as, the only 2-regular graph admitting the PH-property is the cycle on 4 vertices, which happens to be $\mathcal{Q}_2$. We also remark that there is an infinite number of 4-regular graphs having the PH-property (see \cite{AAAHST, GauZer}). Following such a terminology we can state Fink's result from \cite{Fink} as follows.

\begin{theorem}[Fink, \cite{Fink} 2007]\label{fink}
The hypercube $\mathcal{Q}_d$ has the PH-property, for every $d\geq 2$.
\end{theorem}
Recall that the \emph{Cartesian product} $G \Box H$ of two graphs $G$ and $H$ is a graph whose vertex set is $V(G) \times V(H)$, and two vertices $(u_1, v_1)$ and $(u_2, v_2)$ are adjacent precisely if $u_1 = u_2$ and $v_1v_2 \in  E(H)$, or $u_1u_2 \in E(G)$ and $v_1 = v_2$.

Given a graph $G$, the \emph{prism operator} $\mathcal{P}(G) =G \Box K_2$ is the Cartesian product of $G$ with $K_2$. Note that it consists of two copies $G_1$ and $G_2$ of $G$ with the same vertex labelling as in $G$, and an edge between the vertices having the same label. Throughout the paper, we refer to these subgraphs, $G_1$ and $G_2$, as the main copies of $G$ in $\mathcal{P}(G)$. The result of a single application of the operator is usually called the \emph{prism graph} $\mathcal{P}(G)$ of $G$ (see \cite{PrismGraphs}), and repeated applications shall be denoted by powers, with $\mathcal{P}^k(G)$ being the prism graph of $\mathcal{P}^{k-1}(G)$. If needed we shall assume that $\mathcal{P}^0(G)=G$.

It is worth noting that for $d\geq 2$, $\mathcal{Q}_d=\mathcal{P}^{d-2}(Q_2)$. Hence, Theorem \ref{fink} is equivalent to saying that for each $k\geq 0$,  $\mathcal{P}^k(\mathcal{Q}_2)$ admits the PH-property. One might wonder whether it is possible to replace $\mathcal{Q}_2$ with some other initial graph. The main contribution of this paper is Theorem \ref{GK2}, which  generalises Theorem \ref{fink}. We obtain a much larger class of graphs with the PH-property by proving that for every graph $G$ having the PH-property, the graph $\mathcal{P}^k(G)$ has the PH-property for each $k\ge 0$. Hence, Kreweras' Conjecture, and therefore Theorem \ref{fink}, turn out to be special consequences of Theorem \ref{GK2} obtained starting from $G=\mathcal{Q}_2$, which is trivially PH.

Other results on this topic, dealing with the Cartesian product of graphs, were also obtained in \cite{AAAHST} and \cite{GauZer}. In particular, we state the following theorem which shall be needed in Section \ref{sec:iteratedproducts}.

\begin{theorem}[Alahmadi \emph{et al.}, 2015 \cite{AAAHST}]\label{thAAA}
Let $P_q$ be a path of length $q$. The graph $P_q\Box \mathcal{Q}_d$ admits the PH-property, for $d \geq 5$.
\end{theorem}

The above theorem is stated as Theorem 5 in \cite{AAAHST}, where some other results apart from the statement above are proved. We use this result to obtain a similar one for every connected graph $G$ (see Theorem \ref{thm:GboxQd}). More precisely, we prove that for every arbitrary connected graph $G$, the graph $\mathcal{P}^k(G)$ has the PH-property for a sufficiently large $k$, depending on the minimum number of leaves over all spanning trees of $G$. We refer the reader to \cite{AGZ-Rook} and \cite{betwixt} for other papers dealing with the Pairing-Hamiltonian property and related concepts under some graph operations.

\section{Generalising Fink's result}

As stated in the introduction, this section shall be devoted to generalising Theorem \ref{fink}.

\begin{theorem}\label{GK2}
Let $G$ be a graph having the PH-property. Then, for each $k\geq0$, $\mathcal{P}^k(G)$ admits the PH-property.
\end{theorem}

\begin{proof}
Consider $\mathcal{P}(G)$ and let $G_1$ and $G_2$ be the two main copies of the graph $G$ in $\mathcal{P}(G)$. Then, a pairing  $P$ of $\mathcal{P}(G)$ can be partitioned into three subsets $P_1 \cup P_2 \cup X$ where:

\begin{linenomath}
	$$P_i=\{xy \in P~|~\{x,y\} \subset V(G_i),~ \textrm{for each }i\in\{1,2\}\}; \textrm{ and}$$
	$$X=\{xy \in P~|~x \in V(G_1),~y \in V(G_2)\}.$$
\end{linenomath}
Note that $|X| \equiv 0\pmod 2$ since each $G_i$ admits the PH-property and so are both of even order. We shall distinguish between two cases: whether $X$ is empty or not.\\

\noindent\textbf{Case 1.} $|X|=0$.\

\begin{figure}[h!]
    \centering
    \includegraphics[width=.6\textwidth]{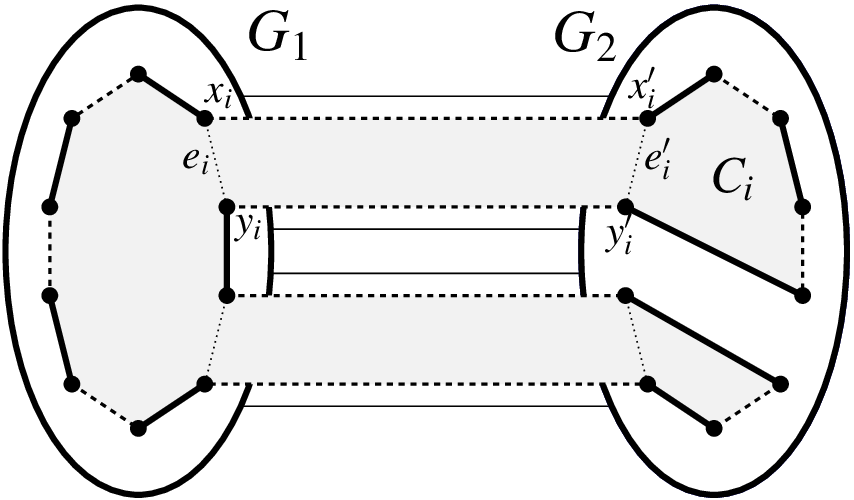} 
    \caption{An extension of the pairing $P$, depicted in bold, when $|X|=0$. The dashed edges represent those in $N$, whereas the dotted edges are the edges $e_i$ and $e'_i$.}
    \label{fig:1b}
\end{figure}

In this case, $P=P_1 \cup P_2$. Since $G_1$ has the PH-property, there exists a perfect matching $M$ of $G_1$ such that $P_1 \cup M$ is a Hamiltonian cycle of $K_{G_1}$. We shall denote the corresponding vertex of $x \in V(G_1)$ by $x' \in V(G_2)$. Let $M'$ be the perfect matching of $G_2$ such that $x'y' \in M'$ if and only if $xy \in M$. In other words, $M'$ is the copy of $M$ in $G_2$. We observe that $P_2 \cup M'$ consists of the union of cycles of even length, say $C_1,\dots , C_t$. Note that cycles of length 2 shall be allowed in the sequel as they arise when $P_2 \cap M' \neq \emptyset$. For each $i \in \{1,\dots,t\}$, we choose an edge $e_i'=x_i'y_i' \in M' \cap C_i$ and we denote the corresponding edge in $M$ by $e_i=x_iy_i$. Consequently, the set 
\begin{linenomath}
$$N=(M \setminus \{ e_1,\ldots, e_t\}) \cup (M' \setminus \{e'_1,\ldots, e'_t\}) \cup \{ x_ix_i',y_iy_i'~|~i\in\{1,\dots,t\} \}
$$
\end{linenomath}
is a perfect matching of $\mathcal{P}(G)$ such that $P \cup N$ is a Hamiltonian cycle of $K_{\mathcal{P}(G)}$. We note that the vertex $x_i'$ in $G_2$ corresponds to the vertex $x_i$ in $G_1$, see Figure \ref{fig:1b}.\\

\noindent\textbf{Case 2.} $|X|=2r>0$.

\begin{figure}[h!]
    \centering
    \includegraphics[width=.6\textwidth]{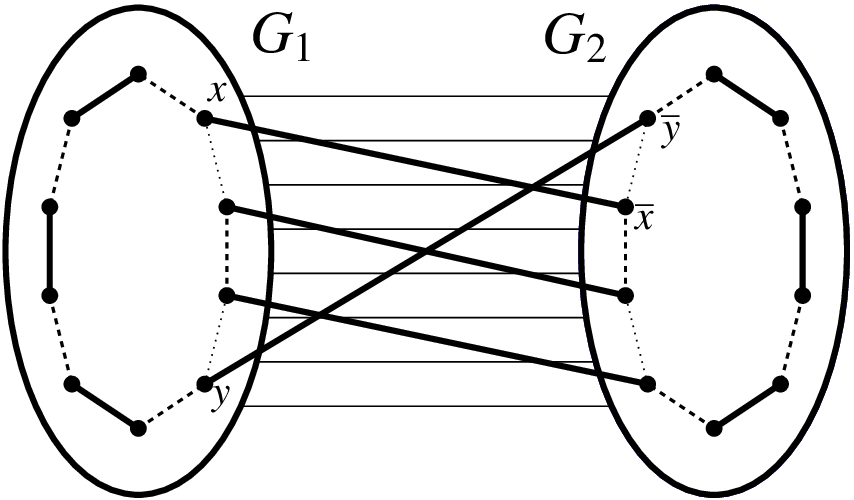} 
    \caption{An extension of the pairing $P$, depicted in bold, when $|X|= 2r>0$. The dashed edges represent those in $M$ and $M'$, whereas the dotted edges are those in $L$ and $R$.}
    \label{fig:2}
\end{figure}

In this case we consider an analogous argument to the one used by Fink to prove Theorem \ref{fink}. Since $|X| \neq 0$, $P_1$ is a matching of $K_{G_1}$ which is not perfect, as there are $2r$ unmatched vertices. Let $L$ be an arbitrary set of $r$ edges of $K_{G_1}$ such that $P_1 \cup L$ is a pairing of $G_1$. Since $G_1$ has the PH-property, there exists a perfect matching $M$, of $G_1$, such that $P_1 \cup L \cup M$ is a Hamiltonian cycle of $K_{G_1}$. Next we define the following set

\begin{linenomath}
    \[
  R = \left\{ \overline{x}~\overline{y}\in E(K_{G_2})\ \middle\vert \begin{array}{l}
    \exists\ x,y \in V(G_1) \text{ with } \{x\overline{x},y\overline{y}\} \subseteq X  \text{ and}\\
     \exists \text{ an } (x,y) \text{-path}   \text{ contained in }P_1 \cup M
  \end{array}\right\},
\]
\end{linenomath}
such that $P_2 \cup R$ is a pairing of $G_2$. Note that $x~\overline{x}$ and $y~\overline{y}$ are edges in $K_G$ since $|X| \neq 0$, and their endvertices might not be corresponding vertices in $G_1$ and $G_2$, as they were in the former case.
Since $G_2$ has the PH-property there exists a perfect matching $M^{\prime}$ of $G_2$, such that $P_2 \cup R \cup M^{\prime}$ is a Hamiltonian cycle of $G_2$. It follows that $P_1 \cup P_2 \cup X \cup M \cup M^{\prime}$ is a Hamiltonian cycle of $K_{\mathcal{P}(G)}$ in which $M \cup M^{\prime}$ is a perfect matching of $\mathcal{P}(G)$, see Figure \ref{fig:2}. 

This proves that  $\mathcal{P}(G)$ has the PH-property and thus, by iterating the prism operator, the result follows.
\end{proof}




\section{Convergence of general graph prisms to the PH-property}\label{sec:iteratedproducts} 
In this section we show that given any connected graph $G$, there exist a sufficiently large integer $k$ such that $\mathcal{P}^k(G)$ has the PH-property. In other words, after iterating the prism operator a sufficient number of times, the resulting graph will have the PH-property. We remark that if a graph contains a spanning subgraph admitting the PH-property, then the graph itself admits the PH-property. Hence, by Theorem \ref{thAAA}, the next corollary follows.

\begin{corollary}
Let $G$ be a traceable graph. For $k \geq 5$, the graph $\mathcal{P}^k(G)$ has the PH-property.
\end{corollary}

Recall that a \emph{traceable} graph is a graph admitting a Hamiltonian path. Next, we show that starting from an arbitrary connected graph $G$, we can always obtain a traceable graph by iterating the prism operator a suitable number of times. To this purpose, we need the following definition and lemma.
\begin{definition}
Let $G$ be a connected graph. The \emph{minimum leaf number} of $G$, denoted by $\textrm{ml}(G)$, is the minimum number of leaves over all spanning trees of $G$.
\end{definition}

Clearly, for any connected graph $G$, $\textrm{ml}(G)\geq 2$, and $\textrm{ml}(G)=2$ if and only if $G$ is traceable.

\begin{lemma}\label{lemma1} 
Let $G$ be a connected graph with $\textrm{\emph{ml}}(G) >2$. Then, $\textrm{\emph{ml}}(G) > \textrm{\emph{ml}}(\mathcal{P}(G))$.
\end{lemma}

\begin{proof}
Suppose that $\textrm{ml}(G) =t>2$ and let $G_1$ and $G_2$ be the two copies of $G$ in $\mathcal{P}(G)$. Let $R_1,R_2$ be two copies of a spanning tree of $G$ with $t$ leaves in $G_1$ and $G_2$, respectively. Let $S=\{e_0,e_1,\dots,e_{t-1}\}$ be the set consisting of the $t$ edges which connect a leaf of $R_1$ to the corresponding leaf of $R_2$. Consequently, we have that $T_0=(R_1 \cup R_2) + e_0$ is a spanning tree of $\mathcal{P}(G)$ with $2t-2$ leaves. Moreover, $T_0+e_1$ has exactly one cycle, say $C_1$. Since $\textrm{ml}(G) >2$, $C_1$ is a proper subgraph of $T_0 +e_1$ and there exists a vertex $v$ of $C_1$ such that $deg_{T_0+e_1}(v) >2$. We note that the removal of an edge of $C_1$, say $f_1$, which is incident to $v$ gives rise to a spanning tree $T_1=T_0+e_1-f_1$ of $\mathcal{P}(G)$ with at most $2t-3$ leaves. Then, for every $j\in \{2,\ldots, t-1\}$, starting from $j=2$ and continuing consecutively up to $t-1$, we choose an edge $f_j$ from $E(T_{j-1}+e_j)$ lying on the unique cycle in $T_{j-1}+e_j$ and incident to a vertex of degree at least 3 in $T_{j-1}+e_j$. We then let $T_j$ be equal to $T_{j-1}+e_j-f_j$, which by a similar argument to the above is a spanning tree of $\mathcal{P}(G)$ with at most $2t-2-j$ leaves. Therefore, $T_{t-1}$ has at most $t-1$ leaves and $\textrm{ml}(\mathcal{P}(G)) \leq t-1 < \textrm{ml}(G)$.
\end{proof}	

From the above statements, it is easy to obtain the following result.

\begin{proposition}\label{PHtrac}
Let $G$ be a connected graph. Then, $\mathcal{P}^k(G)$ is traceable for all $k \geq \textrm{\emph{ml}}(G)-2$.
 \end{proposition}
 
\begin{proof}
If we start from $G$ and apply the prism operator $\textrm{ml}(G)-2$ times, by Lemma \ref{lemma1}, the graph $\mathcal{P}^{\textrm{ml}(G)-2}(G)$ has $\textrm{ml}(\mathcal{P}^{\textrm{ml}(G)-2}(G))=2$.
 Consequently, it admits a Hamiltonian path.
\end{proof}

Combining Theorem \ref{thAAA} and Proposition \ref{PHtrac} we obtain the following.

\begin{theorem}\label{thm:GboxQd}
Let $G$ be a connected graph with $m=\textrm{\emph{ml}}(G)$, then $\mathcal{P}^{m+3}(G)$ has the PH-property.
\end{theorem}

\begin{proof}
If $G$ is traceable, then $m=2$, and so, from Theorem \ref{thAAA} we have that $\mathcal{P}^5(G)$ has the PH-property. On the other hand, if $G$ is not traceable, then $m>2$. By Theorem \ref{PHtrac}, the graph $\mathcal{P}^{m-2}(G)$ is traceable. Hence, by Theorem \ref{thAAA}, $\mathcal{P}^{5}(\mathcal{P}^{m-2}(G))=\mathcal{P}^{m+3}(G)$ admits the PH-property.
\end{proof}

Let us note that the condition given in Proposition \ref{PHtrac}, and subsequently in Theorem \ref{thm:GboxQd}, is only a sufficient condition, though we do not believe it is necessary for sufficiently large values of $\textrm{ml}(G)$. Establishing an optimal lower bound for $k$ in Proposition \ref{PHtrac} in terms of $\textrm{ml}(G)$ would be an interesting problem to consider.

\section{Final remarks}

Several open problems were posed in \cite{AAAHST}. In particular, proving that the graph $P_q \Box \mathcal{Q}_d$ has the PH-property for $d=3,4$ and an arbitrary $q$ is still open. It is dutiful to note that we are aware that in case of a positive answer, Theorem \ref{thm:GboxQd} should be refined accordingly. 

A much more ambitious problem is to wonder whether it is enough for two graphs $G$ and $H$ to have the PH-property, for $G \Box H$ to have the PH-property as well. 
This latter question seems very difficult to prove. Here, we have shown, in Theorem \ref{GK2}, that it holds when $H$ is the hypercube, which is an iteration of the prism operator. In Theorem \ref{thm:GboxQd}, we see that even if $G$ does not have the PH-property, but is connected, a \emph{large enough} number of iterations of the prism operator make it \emph{converge} to a graph with the PH-property. As a matter of fact, we can define the parameter $\mathfrak{p}(G)$  as the smallest non-negative integer $\mathfrak{p}=\mathfrak{p}(G)$ such that $\mathcal{P}^{\mathfrak{p}}(G)$ admits the PH-property. It trivially follows that $\mathfrak{p}(G)=0$ if and only if $G$ is PH. Henceforth, the parameter $\mathfrak{p}(G)$ can be considered as a measure of how far a graph $G$ is from having the PH-property, with respect to the prism operator. Determining the behaviour of $\mathfrak{p}(G)$ for some special classes of graphs could be of interest in the study of the PH-property. 

We could also wonder if there are other graphs that speed up the convergence to the PH-property under the Cartesian product, or on the other hand if there are other products under which the convergence to the PH-property is guaranteed. It seems so if we consider the strong product of graphs.
The \emph{strong product} $G \boxtimes H$ is a graph whose vertex set is the Cartesian product $V(G) \times V(H)$ of $V(G)$ and $V(H)$, and two vertices $(u_1, v_1)$, $(u_2, v_2)$ are adjacent if and only if they are adjacent in $G \Box H$ or if $u_1u_2\in E(G)$ and $v_1v_2\in E(H)$.
It is trivial that $G \Box H$ is a subgraph of $G \boxtimes H$; hence, if $G\Box H$ has the PH-property, then $G \boxtimes H$ will inherit the same property as well.


A result from \cite{accordions} on accordion graphs easily implies that in the case of  Hamiltonian graphs, only one occurrence of the strong product with $K_2$ is enough to obtain a graph with the PH-property.

\begin{theorem}
Let $G$ be a Hamiltonian graph, then $G \boxtimes K_2$ has the PH-property.
\end{theorem}

This suggests that the strong product may have a convergence to the PH-property in few steps also for general graphs.

\end{document}